\documentclass[12pt]{amsproc}%
\usepackage{amsmath}
\usepackage{amssymb}
\usepackage{amsthm}
\usepackage{latexsym}
\usepackage{amsthm}
\usepackage[T2A]{fontenc}
\usepackage[cp1251]{inputenc}
\RequirePackage{srcltx}

\newcommand{\BR}{\mathbb{{R}}}
\newcommand{\T}{\mathbb{{T}}}
\newcommand{\N}{\mathbb{{N}}}

\newcommand{\D}{\mathbb{{D}}}
\newcommand{\cP}{\mathcal{{P}}}

\RequirePackage{srcltx}

\def\le {\leqslant}
\def\ge {\geqslant}

\usepackage{amssymb}
\usepackage{amsfonts}
\usepackage{amsmath}
\usepackage{graphicx}%

\providecommand{\U}[1]{\protect\rule{.1in}{.1in}}
\theoremstyle{plain}

\newtheorem{theorem}{Theorem}[section]
\newtheorem*{problema}{Problem A}
\newtheorem*{problemb}{Problem B}
\newtheorem*{problemc}{Problem C}
\newtheorem*{problemd}{Problem D}

\newtheorem{lemma}[theorem]{Lemma}
\newtheorem{remark}[theorem]{Remark}

\newtheorem{Definition}{Definition}

\numberwithin{equation}{section}
\def\C {\mathbb{C}}

\def\restr#1{\,\vrule\,\lower1ex\hbox{$#1$}}

\def\ve{\varepsilon}

\def\l{\lambda}

\def\o{\omega}

\def\t{\theta}
\def\z{\zeta}

\begin{document}

\title[Sharp Remez inequality]
{Sharp Remez inequality
}

\author{S. Tikhonov}
\address{S. Tikhonov
\\
Centre de Recerca Matem\`{a}tica\\
Campus de Bellaterra, Edifici C
08193 Bellaterra (Barcelona), Spain\\
ICREA, Pg. Llu\'{i}s Companys 23, 08010 Barcelona, Spain\\
 and Universitat Aut\`{o}noma de Barcelona 08193 Bellaterra (Barcelona), Spain
}
\email{stikhonov@crm.cat}

\author[P.\ Yuditskii]{P. Yuditskii}

\address{P. Yuditskii \\
Abteilung f\"ur Dynamische Systeme und Approximationstheorie, Johannes Kepler Universit\"at Linz, A-4040 Linz, Austria}

\email{Petro.Yudytskiy@jku.at}

\date{September 17, 2018}

\begin{abstract}
Let an {algebraic} polynomial $P_n(\zeta)$  of degree $n$ be such that
$|P_n(\zeta)|\le 1$ for  $\zeta\in E\subset\T$ and $|E|\ge 2\pi -s$. We prove the sharp Remez inequality
$$
\sup_{\zeta\in\T}|P_n(\zeta)|\le\mathfrak{T}_{n}\left(\sec \frac{s} 4\right),
$$
where $\mathfrak{T}_{n}$ is the Chebyshev polynomial of degree $n$.
The equality holds if and only if
$$
P_n(e^{iz})=e^{i(nz/2+c_1)}\mathfrak{T}_n\left(\sec\frac s 4\cos \frac {z-c_0} 2\right), \quad c_0,c_1\in\BR.
$$
This gives the solution of the long-standing problem on the sharp constant in the Remez  inequality for trigonometric polynomials.

\end{abstract}
\maketitle

\vskip 0.6cm

\section{Introduction}
Let
$\mathfrak{T}_n(x)$ be the Chebyshev polynomial of degree $n$, i.e.,
$$\mathfrak{T}_n(x)=\frac{1}{2}\Big(\big(x+\sqrt{x^2-1}\big)^n+\big(x-\sqrt{x^2-1}\big)^n\Big)$$ for every $x\in \mathbb{R}$
and
 $|B|$ denote the  Lebesgue measure of a measurable set $B$.

The Remez inequality \cite{re}  for algebraic polynomials $P_n$ asserts that
$$
\max_{x\in[-1,1]} |P_n(x)|\le \mathfrak{T}_n\left(\frac{2+s}{2-s}\right),$$
for every $P_n$ satisfying
$$
\Big|\{x:[-1,1]: |P_n(x)|\le 1\}\Big|\ge 2-s, \qquad 0<s<2.$$
Moreover, the
equality holds if and only if
$$P_n(x)=
\pm
\mathfrak{T}_n\left(\frac{\pm 2x+s}{2-s}\right).
$$
The Remez inequality 
 plays an important role in many problems in approximation theory, harmonic and functional analysis
  (see, for example, \cite{bor, brud, yom, totik} and references therein).

A similar question for trigonometric polynomials is
 a widely studied and open problem. In more detail,
 consider the space of  (complex) trigonometric polynomials of degree at most $n\in \mathbb{N}$, i.e.,
\begin{eqnarray*}
\frak N_n=\Big\{Q_n\,:\quad Q_n(x)=\sum_{|k|\le n} c_ke^{ikx},\quad
c_k\in \mathbb{C}, \quad x\in [0,2\pi)
\Big\}.
\end{eqnarray*}
\begin{problema}
 How large can $\|Q_n\|_{L_\infty([0,2\pi))}$ be if
\begin{equation}\label{remez-33}
\left|\Big\{x\in [0,2\pi): |Q_n(x)|\le 1\Big\}\right|\ge 2\pi- s
\end{equation}
holds
for some $0<s< 2\pi$?
\end{problema}
In other words, we study the
best constant in the Remez inequality
\begin{equation}\label{rem-or}
\|Q_n\|_{C[0,2\pi)}\le C(n,s)
\|Q_n\|_{C([0,2\pi) \setminus B)}, \qquad {Q}_n\in \frak N_n,\qquad |B|=s.
\end{equation}
The problem of finding the sharp constant in the Remez inequality, or at least of obtaining some suitable bounds on this constant,  has been extensively  studied starting from  the 1990s 
 (see, e.g., \cite{an, an1, bor, er, er3, er4, ga, gunz, kroo, nurs}) but some estimates were established  much earlier \cite{stul, turan}.

The sharp constant in (\ref{rem-or}) is known only in the case of  an interval, that is,
(see, e.g., \cite{er})
\begin{equation}\label{rem-or-or}
\|Q_n\|_{C[0,2\pi)}\le
\mathfrak{T}_{2n}\Big(\csc \frac a2\Big)
\|Q_n\|_{C[-a,a] },
 \qquad
 [-a,a] \subset [-\pi,\pi].
\end{equation}
 In the general case the problem remains open.  Erd\'{e}lyi \cite{er} proved that $$C(n,s)\le \exp({4ns}), \qquad s\le\frac\pi 2.$$
 Later on, this  result was slightly improved in \cite{ga} as follows $C(n,s)\le\exp({2ns})$.
Moreover, it is known \cite{nurs} that, 
 for a fixed $n$,
$$
 {C}(n, s)= 1+\frac{(ns)^2}{8} +O(s^4)\qquad\mbox{as}\quad s\to 0.
$$

Several papers studied the behavior of the   constant $C(n,s)$ in (\ref{rem-or}) for $B$ with large measure ($|B|=s\ge\frac\pi 2$), see \cite{er3, ga, nazarov}.
In this case, we have
 $$C(n,|B|)\le\left(\frac{A}{2\pi-s}\right)^{2n},$$
 where $A$ is an absolute constant which can be taken  as $17$.



The solution of  the aforementioned  Problem A is given as follows.
\begin{theorem}\label{original}
Let $Q\in \frak {N}_{n}$ be such that
\eqref{remez-33} holds for some $0<s< 2\pi$.
Then
\begin{equation}\label{remez-estimate}
\|Q\|_{C[0,2\pi)}\le
\mathfrak{T}_{2n}\Big(\sec \frac s4\Big).
\end{equation}
Moreover the equality holds if and only if
$$
Q(x)=e^{ic_1}\mathfrak T_{2n}\left(\sec \frac s 4\cos\frac {x-c_0} 2\right), \quad c_0, c_1\in\BR.
$$
\end{theorem}
Estimate (\ref{remez-estimate}) confirms a conjecture by Erd\'{e}lyi  \cite{er3}, also stated in \cite{ga}.

 The proof of Theorem \ref{original} is  based on a solution of the following more general problem.

\begin{problemb}
Let $$P_n(\zeta)= A_0+A_1\zeta+\cdots+A_n\zeta^n, \qquad \zeta=e^{iz},$$
and
$$ \mathcal{P}_n(E):=\Big\{P_n(\z):\
|P_n(\zeta)|\le 1, \zeta\in E\Big\},
$$
where $E\subset \T$ is a measurable set.
For a fixed $s\in (0,2\pi)$ find
$$
\sup\{\max\limits_{{}^{P_n\in \mathcal{P}_n(E)}_{c \in [0,2\pi)}
} |P_n(e^{ic})|: |E|\ge 2\pi-s\}. 
$$
\end{problemb}
Since for a trigonometric polynomial
$Q_{n}(x)=\sum_{-n\le k\le n} c_ke^{ikx}$ one has
$Q_{n}(x) e^{in x}=P_ {2n}(e^{ix})$ with $\|Q_{n}\|_{C}= \|P_ {2n}\|_{C}
$ and
 $|Q_{n}(x_0)|= |P_ {2n}(e^{ix_0})|$,
where
$$
\max\limits_{{}^{P_n\in \mathcal{P}_n(E)}_{c \in [0,2\pi)}
} |P_{2n}(e^{ic})|=|P_ {2n}(e^{ix_0})|,
$$
 Theorem \ref{original} will follow from the corresponding solution of Problem B
 for even-degree polynomials.
 For the partial case of  even trigonometric polynomials,  i.e.,  having the form $P_n(\cos t)$, see also the recent preprint \cite{er44}.
  We point out that we solve Problem B \textit{for all integer} $n$.

\smallskip

Now we can state our main result.
\begin{theorem}\label{thmain}
Let $P_n(\zeta)$ be an {algebraic} polynomial of degree $n$ bounded by one on a subset $E$ of the unit circle $\T$. If $|E|\ge 2\pi -s$
for some $0<s< 2\pi$, then
\begin{equation}\label{15sep1}
\sup_{\zeta\in\T}|P_n(\zeta)|\le\mathfrak{T}_{n}\left(\sec \frac{s} 4\right).
\end{equation}
Moreover, the equality holds if and only if
\begin{equation}\label{16sep3}
P_n(e^{iz})=e^{i(nz/2+c_1)}\mathfrak{T}_n\left(\sec\frac s 4\cos \frac {z-c_0} 2\right), \quad c_0,c_1\in\BR.
\end{equation}

\end{theorem}

We reduce 
 Problem B to Problems C and D below.
\begin{problemc}
Let $E$ be a closed proper subset of $\T$,
\begin{equation}\label{setE}
E=\T\backslash  \bigcup_{j=0}^g (e^{ia_j},e^{ib_j}),\quad 0\le g\le\infty.
\end{equation}
Let
$$
P_n(\zeta)= A_0+A_1\zeta+\cdots+A_n\zeta^n, \qquad \zeta=e^{iz},
$$
and
$$
\mathcal{P}_n(E):=\Big\{P_n(\z):\
|P_n(\zeta)|\le 1, \zeta\in E\Big\}.
$$
Find
$$
\max\limits_{{}^{P_n\in \mathcal{P}_n(E)}_{c \in (a_0,b_0)}
} |P_n(e^{ic})|. 
$$
\end{problemc}

\begin{problemd}
Let $E$ and $\cP_n(E)$ be given as in Problem C.
Find
$$
\max\limits_{P_n\in \mathcal{P}_n(E)} |P_n(e^{ic})|,\qquad c\in(a_0,b_0),  \quad\mbox{$c$\, is fixed}.
$$
\end{problemd}

In the next Section we
  reveal the structure of  extremal polynomials in Problem D.  We describe them by means of  certain conformal mappings on  the  comb domains,
  see Theorem \ref{th2}.
  Particularly, this allows us to define the so-called $n$-regular extension for a given set $E$
 and to give a relation between the solutions of Problems C and D.
  In Section 3 we prove our main theorem. First, in Lemma \ref{lemma} we show that an extremal configuration of a set $E$  for Problem B belongs to the class of  $n$-regular sets (a set which coincides with its $n$-regular extension). Second, in Lemma \ref{lemma2} we prove that among these sets a single-arc set is   extremal. Finally, we obtain an explicit formula for extremal polynomials using some elementary conformal mappings.

\vskip 0.6cm

\section{Comb domains and solutions of Problems D and C}
Comb-domains and the corresponding conformal mappings were introduced by Akhiezer and Levin, see  \cite{levin}.
Nowadays  they are actively used  in spectral theory
(see, e.g., \cite{mar}) and approximation theory
(see, e.g., \cite{yu3}). For a modern presentation, see
\cite{yu2}.
In this paper we will employ  only the 
 periodic $n$-regular comb domains.

\begin{Definition}
Let $g\in\N$ and  $\{h_k\}_{k=0}^g$ be a  collection of positive numbers.
The periodic $n$-regular comb domain is of the form
\begin{eqnarray*}
\Pi(h_0,\cdots,h_g)&=&\Pi(h_0,\cdots,h_g;\o_0,\cdots,\o_g)\\&=&\C_+ \backslash \bigcup_{k=0}^g \bigcup_{m\in \mathbb{Z}} \big\{z=\omega_k+2\pi m+iy\;\, (0<y\le h_k)
\big\},
\end{eqnarray*}
where
$\omega_k =2\pi\frac{j_k} n$, $0\le j_k<n$; see Fig. 1.
\end{Definition}

\begin{figure}[htbp]
\begin{center}
\includegraphics[scale=0.3]{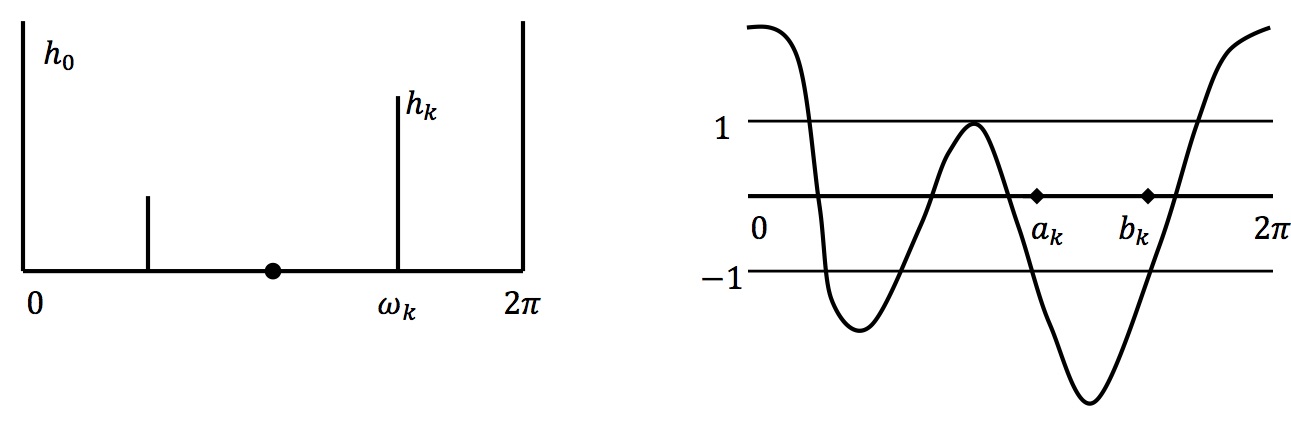}
\caption{Comb domain and graph of the function $\cos\frac n 2\theta(z)$ on the period}
\end{center}
\end{figure}

For $\Pi=\Pi(h_0,\cdots,h_g)$, consider the conformal mapping $\theta$  defined as follows:
\begin{equation}\label{theta1}
\theta: \C_+ \longrightarrow \Pi,\qquad
\theta(iy)\sim iy\quad\mbox{as}\quad y\to \infty.
\end{equation}
We say that a set $E$ is $n$-regular if $$E=\Big\{e^{iz}:z\in \theta^{-1}(\mathbb{R})\Big\},$$
where $\theta$ is given by (\ref{theta1}). Note that in this case
the  endpoints of the gaps in (\ref{setE}) are given by
\begin{eqnarray*}
&&a_k=\theta^{-1}(\omega_k-0),  \qquad
0\le k\le g;
\\&&b_k=\theta^{-1}(\omega_k+0),  \qquad
0\le k\le g.
\end{eqnarray*}


\begin{theorem} \label{th2}
For an  extremizer $T_n(e^{iz},e^{ic_0}, E)$ of Problem D, $c_0\in(a_0,b_0)$, there exists
the representation
\begin{equation}\label{16sep7}
T_n(e^{iz})=e^{\frac{inz}{2}}\cos\frac n2\theta(z),
\end{equation}
where $\theta$ is  a conformal mapping $\C_+ \to \Pi$ with the
normalization
\begin{eqnarray}\label{vsp}
\theta(iy)\sim iy \qquad \mbox{as}\quad y\to \infty.
  \end{eqnarray}
Moreover,
\begin{eqnarray}\label{vsp+}
\theta(a_0)=0-,
\qquad\theta(b_0)=0+.
  \end{eqnarray}
Consequently,   the extremizer does not depend on a position of $c_0$ in the given gap $(a_0,b_0)$.
\end{theorem}

\begin{proof}
The proof is based on  Markov's method of correction of an extremal function
\cite[Chapter 7]{yu3}, see also  \cite[Theorem 3.2]{ahlfors}.

Let $c_0\in (a_0,b_0)$ and
 $T_n(e^{iz})=T_n(e^{iz},e^{ic_0}, E)$ be an extremizer  in Problem D.
We write it as follows
$$
T_n(e^{iz})=e^{i\frac{zn}{2}}F(z),
$$
where $F$ is an entire function such that $F(z+2\pi)=(-1)^{n}F(z).$ Without loss of generality
we assume that $F(c_0)>0$.
Noting that the function $e^{i\frac{zn}{2}}(F(z)+\overline{F(\overline{z})})/2$ is also an extremizer  in Problem D,
we may suppose that
$F(z)\in \mathbb{R}$ whenever  $z\in\mathbb {R}$.
We divide the rest of the proof into 6 steps.

Step 1. We show that $F$ has no complex zeros.
Assume that $F(z_0)=0$, Im $z_0>0$, which gives that $F(\overline{z_0})=0$.
Define the Markov correction function $(\delta>0)$
$$
Q(\zeta)=T_n(\zeta)\left(
1-\delta\frac{( \zeta-e^{ic_0})( 1-\zeta e^{-i{c_0}})}{( \zeta-e^{iz_0})( 1-\zeta e^{-i\overline{z_0}})}\right),
$$
which is a polynomial of degree at most $n$ satisfying
$T_n(e^{ic_0})=Q(e^{ic_0})$.
Since the expression in the brackets is equal to
$$
1-\delta
\Big|\frac{\zeta-e^{ic_0}}{ \zeta-e^{iz_0}}
\Big|^2$$ whenever $\zeta\in \T$,
 we have
$\max_{\zeta\in E} |T_n(\zeta)|
>
\max_{\zeta\in E} |Q(\zeta)|
$
for small enough  $\delta$.
 This contradicts that $T_n$ is an extremizer.

Step 2. We prove that zeros of $F$ are simple.
Here we use the similar correction function $(\delta>0)$
$$
Q(\zeta)=T_n(\zeta)\left(
1-\delta\frac{( \zeta-e^{ic_0})( 1-\zeta e^{-i{c_0}})}{( \zeta-e^{iz_0})( 1-\zeta e^{-i{z_0}})}\right)
$$
and follow the proof of the previous step.

Step 3. Let us show that between two consecutive zeros of $F$ (take, for example, $F(z_1)=F(z_2)=0$)
  there is $y$ such that $|F(y)|=1$ and $e^{iy}\in E$.
Assume the reverse:
$$
\text{(i)}\ \{e^{iy}: y\in (z_1,z_2)\}\cap E=\varnothing\quad \text{or}\quad
\text{(ii)}\ \max_{y\in (z_1,z_2)}|F(y)|<1.
$$ Then we
define the correction function   $(\delta>0)$
$$
Q(\zeta)=e^{i\frac{nz}{2}}G(\zeta),
\qquad
G(\zeta)=F(\zeta)\left(
1-\delta\frac{ \sin^2\frac{z-c_0}{2}}{\sin\frac{z-z_1}{2}\sin\frac{z-z_2}{2}}\right).
$$
On the set $I_1=(z_1-\varepsilon, z_1+\varepsilon)\cup (z_2-\varepsilon, z_2+\varepsilon)$
taking sufficiently small $\varepsilon$ and $\delta$, we have $|G(z)|<1.$
In the case (i) the function $|G(z)|$ on the set $I_2=(z_1+\varepsilon, z_2-\varepsilon),$ should not be restricted. In the case (ii) since $\max_{z\in I_2} |F(z)|$ has a fixed value
less then one,
for sufficiently small $\delta$ we obtain $|G(z)|<1.$
Finally, on the set
$I_3=[-\pi,\pi]\backslash (I_1\cup I_2)$ we always have $\max_{z\in I_3} |G(z)|<1$.
Then
$\max_{\zeta\in E} |T_n(\zeta)|
>
\max_{\zeta\in E} |Q(\zeta)|
$
gives a contradiction.

Step 4. The following condition holds:
 $F(a_0)=F(b_0)=1$. In our normalization this corresponds to \eqref{vsp+}.
Here we assume that $F(b_0)<1$ and take the first zero $F(\xi)=0$, $\xi>b_0$.
Considering the function $(\delta>0)$
$$
G(\zeta)=F(\zeta)\left(
1-\delta\frac{ \sin\frac{z-c_0}{2}}{\sin\frac{z-\xi}{2}}\right)
$$
implies a contradiction.

Step 5. Let $\{z_k\}_{k=1}^n$ be zeros of the function $F(z)$ in a period, respectively, $\{e^{iz_k}\}_{k=1}^n$
are zeros of $T_n(\z)$. Between each two of them (including the pair $(z_n,z_{n+1})$, with the agreement $z_{n+1}:=z_1+2\pi$) there is at least one critical point $y_k$ such that $F'(y_k)=0$.
Since
$$
F'(z)=e^{i(-n/2)z}\left(i\z T'_n(\z)-i\frac n 2 T_n(\z)\right),
$$
the total number of these points in a period is at most $n$.
Thus all critical points of the function $F(z)$ are real. Also, by step 3, in each interval $(z_k,z_{k+1})$ there  is a point
$x_k$, $e^{ix_k}\in E$, such that
$|F(x_k)|=1$. Therefore at the critical point we have $|F(y_k)|\ge 1$. An example 
  of such a function is shown in Fig. 1. According to the Marchenko--Ostrovskii theorem,
see, e.g., \cite[Section 7.3]{yu3}, the function $F(z)$ possesses the comb representation $F(z)=\cos \frac n 2\t(z)$ with $\t(z)$ normalized by \eqref{vsp}. Since $F(z)$ is periodic, $F(z+2\pi)=(-1)^n F(z)$, the corresponding comb domain is periodic. That is, \eqref{16sep7} is proved.

Step 6.  We prove the last assertion of the theorem. Let $c$ be an arbitrary point in $(a_0,b_0)$. We will show that
\begin{equation}\label{22sep1}
|P_n(e^{ic})|\le |T_n(e^{ic})|,\quad \forall P_n(\z)\in \cP_n(E).
\end{equation}
First we note that there is a polynomial $Q_n(\z)\in\cP_n(E)$  of the form
$$
Q_n(e^{iz})=e^{inz/2}G(z),\quad G(z)=\overline{G(\overline{z})}
$$
such that $|Q_n(e^{ic})|=|P_n(e^{ic})|$. For an arbitrary positive $\ve$ consider the function
$$
H_\ve(z)=(1+\ve)F(z)-G(z).
$$
This function attains 
  positive and negative values at points $\{x_k\}_{k=1}^n$, since $F(x_k)=\pm 1$, see step 5. Therefore each interval $(x_k,x_{k+1})$ (with our  agreement
$x_{n+1}:=x_1+2\pi$) contains at most one zero of $H_{\ve}(z)$. In particular, the interval $(a_0,b_0)$ may have at most one zero of this function.
But, due to step 4,
$$
H_{\ve}(a_0)=(1+\ve)-G(a_0)>0 \quad\text{and}\quad H_\ve(b_0)=(1+\ve)-G(a_0)>0.
$$
Thus, for all $\ve>0$,  $H_\ve(z)$ is positive on this interval and we obtain
$$
|P_n(e^{ic})|=|Q_n(e^{ic})|= |G(e^{ic})|\le F(e^{ic})=|T_n(e^{ic})|,
$$
that is,
\eqref{22sep1} holds.

\end{proof}

\begin{remark}\label{rem22a}
In fact, we proved that
$$
T_n(e^{ic}, e^{ic_0},E)=e^{inc/2}\sup_{P_n\in\cP_n(E)}|P_n(e^{ic})|,\quad \forall c\in(a_0,b_0),
$$
which also implies the uniqueness of the extremal polynomial.
Since the extremizer does not depend on the position of $c_0\in (a_0,b_0)$, in what follows we  write $T_n(\z,(a_0,b_0),E)$ instead of
$T_n(\z,e^{ic_0},E)$.

Moreover, from the proof we see that
a polynomial of the presented form \eqref{16sep7}
 is extremal
on
every set $\tilde E$, which contains at least one of possibly two different points
$e^{i\theta^{-1}(2\pi k/n\pm 0)}$ for $k=1,\cdots,n-1$, together with the endpoints $e^{ia_0}$ and $e^{ib_0}$. Such collections of points  form the so-called maximal  Chebyshev sets \cite[Section 7.2]{yu3}
 for the given extremal function. Note also the fact that every periodic comb generates a periodic function is shown in \cite[Appendix A]{damanik}.
 \end{remark}

\begin{Definition}
Let $E$ be a closed proper subset of $\T$. Let $T_n(\z,(a_0,b_0),E)$ be the extremizer associated to the gap $(a_0,b_0)$
 presented in the form \eqref{16sep7}.
We will call
$$\widehat{E}=\Big\{e^{iz}:z\in \theta^{-1}(\mathbb{R})\Big\}$$
the $n$-regular extension of the set $E$
associated to the gap $(a_0, b_0)$. This is the maximal possible set on which $T_n(\z,(a_0,b_0),E)$ remains extremal in the sense of Problem D.

\end{Definition}

\begin{remark}\label{rem22}
Since the $n$-regular extension does not depend on $c\in (a_0, b_0)$ we obtain a connection between the solutions of Problems C and D.
Let  $c_*=\t^{-1}(ih_0)$. Then
$$
\max\limits_{{}^{P_n\in \mathcal{P}_n(E)}_{c \in (a_0,b_0)}
} |P_n(e^{ic})|=
\max\limits_{P_n\in \mathcal{P}_n(E)} |P_n(e^{ic_*})|=|T_n(e^{ic_*},(a_0,b_0),\widehat{E})|=\cosh \frac n 2 h_0.
$$
Without loss of generality,
we will assume  that $c_*=0$ and impose the third normalization condition for $\theta$ (see (\ref{vsp}))  given
by
\begin{equation}\label{22sep3}
\theta(0)=ih_0.
\end{equation}
\end{remark}

\vskip 0.6cm

\section{Proof of Theorem \ref{thmain}}

We start with the following two lemmas.

\begin{lemma}\label{lemma}
For any
$$
E=\T\backslash  \bigcup_{j=0}^g (e^{ia_j},e^{ib_j}), \qquad 1\le g\le \infty
$$
there is an $n$-regular set
$${E}^*
=\T\backslash  \bigcup_{j=0}^{{g}^*} (e^{i a_j^*},e^{i b_j^*}), \quad 1\le {g}^*< \infty,
$$
 such that
$|{E}^*|=|E|$ and moreover,
$$
\max_{c\in(a_0,b_0)}|T_n (e^{ic}, (a_0,b_0), E)|=|T_n (1, (a_0,b_0), E)|\le |T_n (1, (a_0^*,b_0^*), {E}^*)|.$$
\end{lemma}
\begin{proof}
In the proof we deal with the gap corresponding to $j=0$, so the dependence of $(a_0,b_0)$ will be dropped.

Let $E$ be not $n$-regular. Consider $\widehat{E}$. Since the extension is proper, we have
$|\widehat{E}|>|{E}|$ and moreover, $T_n(\zeta,E)=T_n(\zeta,\widehat{E})$.
Note that $\widehat{E}$ has a finite number of gaps, i.e., $\widehat{g}<\infty$.
This is because  a number of  critical points (where the derivative of $e^{-\frac{inz}{2}}T_n(e^{iz})$ is zero) on a period
is finite
and each gap contains a critical point, see Fig. 1.

Let $\Pi=\Pi(h_0,h_1,\cdots,h_{\widehat{g}})$ be the comb associated to the extremal polynomial $T_n(e^{iz}, \widehat{E})$.
Let also ${\Pi}_h=\Pi(h_0+h,h_1,\cdots,h_{\widehat{g}})$ with $h\ge0$ and
$\widehat{\theta_h}$ be the corresponding conformal mapping normalized exactly as \eqref{vsp}, \eqref{22sep3} and respectively
$T_n(e^{iz},\widehat{E_h})=\cos \frac{n}2\widehat{\theta_h}(z) e^{\frac{inz}{2}}$ (by Theorem \ref{th2}).

We will show that $ |\widehat{E_h}|$ is decreasing with $h$ and tends to zero as $h\to\infty$. In the same time,
$T_n(1,\widehat{E_h})= \cosh \frac n2(h+h_0)$, increasing with $h$.
Therefore, for some $h_*>0$, we will have
$|\widehat{E_{h_*}}|=|E|$ and $\cosh \frac n2(h_*+h_0)>\cosh \frac n2h_0$.
Thus, one can set $E^*=\widehat{E_{h_*}}$.

It is left to verify that
 $ |\widehat{E_h}|$ is decreasing with $h$
and that $ |\widehat{E_h}|\to 0$ as $h\to \infty$.
First, define $w_h(\zeta)=e^{i \widehat{\theta_h}(z)}$, $\zeta=e^{iz}$. Since $\widehat{\theta_h}(z)$ is $2\pi$-periodic
 this map is well-defined.
This is a conformal mapping of the unit disk on the radial slit domain (the unit disk with the system of radial slits), that is,
$$
w_h: \D\to\D\backslash \bigcup_{k=0}^{\widehat{g}} \Big\{w=e^{i\omega_k-y}, \;\; 0<y\le h+h_0, \; 0<y\le h_k, 1\le k\le \widehat{g}\Big\} =:\Omega_h.
$$
According to
the principle of harmonic measure \cite[Chapter 3]{nev},
the harmonic measure $\omega(0,\T, \Omega_h)$ is monotonic with $h$. On the other hand, since $\widehat{E_h}=w_h^{-1}(\T)$ and $w_h(0)=0$, we have
$$\frac1{2\pi}|\widehat{E_h}|=
\omega(0,\widehat{E_h}, \D)
=
\omega(0,\T, \Omega_h).
$$

Second, consider the conformal mapping $\widetilde{\theta_h}:\C_+ \to \Pi_h$, which satisfies another normalization
$$
\widetilde{\theta_h}(-1)=0,\quad \widetilde{\theta_h}(1)=2\pi,\quad
\widetilde{\theta_h}(\infty)=\infty.
$$
We define
$$
-C_-(h):=\widetilde{\theta_h}^{-1}(i(h+h_0)),\qquad
 C_+(h):=\widetilde{\theta_h}^{-1}(2\pi+i(h+h_0)).
$$
By Carath\'{e}odory kernel convergence theorem \cite[p. 28]{pomm},
 in this normalization the sequence of conformal mappings has a limit, and $\lim_{h\to \infty} C_{\pm}(h)=\infty$. Comparing
$\theta$ and $\widetilde{\theta}$, we obtain
$$\frac 1 {2\pi }|\widehat{E_h}|\le \frac{2}{C_-(h)+ C_+(h)}\longrightarrow 0\qquad\mbox{as} \qquad h\longrightarrow\infty.$$

\end{proof}

\begin{lemma}\label{lemma2} Let
$$
E=\T\backslash  \bigcup_{j=0}^g (e^{ia_j},e^{ib_j}), \qquad 1\le g< \infty
$$
be an $n$-regular set.
Then there exists
 $$\widehat{E}=
\T\backslash  \bigcup_{j=0}^{g-1} (e^{i\widehat{a}_j},e^{i\widehat{b}_j})
$$ such that
$|\widehat{E}|=|E|$ and
$|T_n (1, E)|\le |T_n (1, \widehat{E})|$.
\end{lemma}

\begin{proof}

Let $\Pi=\Pi(h_0,h_1,\cdots,h_g)$.
Using again the principle of harmonic measure with respect to the length of the slit $h_g$ for the set
$E_{\check{h}_g}$ associated to the comb
$${\Pi}_{\check{h}_g}=\Pi(h_0,h_1,\cdots,h_{g-1},{\check{h}_g}), \,\,{\check{h}_g}<{{h}_g},$$
we have that
$|{E}_{\check{h}_g}|>|E|$ and
$$|T_n (1, E)|= \cosh \frac n 2 h_0 = |T_n (1, {E}_{\check{h}_g})|.$$

In this way, decreasing  ${\check{h}_g}$, we can delete one of the gaps, i.e., to obtain   ${\check{h}_g}=0$ and, respectively, the set $\check{E}:=E_0$.
Thus, the number of gaps is now one less.

Repeating the proof of Lemma \ref{lemma} with respect to $\check{E}$,
we increase the value of $h$ in $\Pi(h_0+h,h_1,\cdots,h_{g-1})$. Then, finally, we obtain for some $h_*>0$
the comb
$\hat{\Pi}:=\Pi(h_0+h_*,h_1,\cdots,h_{g-1})$
and the set $\widehat{E}$
such that
$|\widehat{E}|=|E|$ and
$|T_n (1, E)|\le |T_n (1, \widehat{E})|$.



\end{proof}

\vskip 0.5cm

\begin{proof}[Proof of Theorem \ref{thmain}]
In light of Lemma \ref{lemma} the extremal configuration corresponds to a regular set $E$. In Lemma \ref{lemma2} we proved that
the extremal configuration corresponds to the case of a single gap,   i.e., to the set of the form
$$
E_s=\T\setminus (e^{-is/2},e^{is/2}),
$$
which corresponds to the comb
$
\Pi=\Pi(h_0).
$
Due to the extremality,  we can claim that for an arbitrary polynomial $P_n(\zeta)$ such that
$$
\left|\left\{ x\in[0,2\pi): |P_n(e^{ix})|\le 1\right\}\right|\ge 2\pi- s,
$$
we have
\begin{equation}\label{15sep5}
\sup_{\zeta\in \T}|P_n(\zeta)|\le |T_n(1,E_s)|=\cosh n\frac {h_0} 2.
\end{equation}
It remains to find the relation between the value $h_0$ and the length of the gap $s$.

In term of the variables
$$
 w(\zeta)=e^{i\theta(z)},\quad \zeta=e^{iz},
$$
we have a conformal mapping of the unit disk $\D$ on the domain
$$
\Omega=\D\setminus\{w=u: u\in (e^{-h_0},1)\}
$$
such that $w(0)=0$ and, by symmetry, $w(\zeta)$ is real for real $\zeta$. We point out that the preimage of the radial slit is the arc
$(e^{-is/2},e^{is/2})$.

By  a standard change of variables
\begin{equation}\label{16sep1}
\lambda=i\frac{1-\zeta}{1+\zeta}, \quad \mu(\lambda)=i\frac{1-w(\zeta)}{1+w(\zeta)},
\end{equation}
we pass to a conformal mapping of the upper half-plane $\C_+$ to the upper half plane with a cut along a single vertical interval
$$
\C_+\setminus\{\mu=i\eta: \eta\in(0,\kappa)\},
$$
where
\begin{equation}\label{15sep3}
\kappa=\frac{1-e^{-h_0}}{1+e^{-h_0}}=\tanh\frac {h_0}2.
\end{equation}
Respectively, the preimage of this vertical interval is the interval $(-\lambda_0,\lambda_0)\subset\BR$, where
\begin{equation}\label{15sep4}
\lambda_0=i\frac{1-e^{is/2}}{1+e^{is/2}}=\tan \frac s 4.
\end{equation}
It is well known (and easy to check directly) that this conformal mapping is of the form
$$
\mu(\lambda)=C\sqrt{\lambda^2-\lambda_0^2},\quad C>0.
$$
To find $C$, we use the normalization condition $\mu(i)=i$. We have
$$
\mu(\lambda)=\frac{\sqrt{\lambda^2-\lambda_0^2}}{\sqrt{1+\lambda_0^2}}
\quad\text{and particularly}\quad
\mu(0)=i\frac{\lambda_0}{\sqrt{1+\lambda_0^2}}=i\kappa.
$$
By \eqref{15sep3} and \eqref{15sep4} we have
$$
\frac{\tan \frac s 4}{\sqrt{1+\tan^2\frac s 4}}=\sin\frac s 4=\tanh \frac {h_0} 2,
$$
or $\cosh \frac {h_0} 2=\sec\frac s 4$. We substitute this value of $h_0$ in \eqref{15sep5}, as the result we obtain \eqref{15sep1}.

Further, since
$$
\mu^2+1=(\l^2+1)\cos^2 \frac s 4,
$$
using \eqref{16sep1}, we obtain
$$
\frac{4 w}{(1+w)^2}=
\frac{ 4 \z}{(1+\z)^2}\cos^2\frac s 4.
$$
Thus
$$
\frac{w^{1/2}+w^{-1/2}}2=\sec\frac s 4 \frac{\z^{1/2}+\z^{-1/2}}2=\sec \frac s 4\cos \frac z 2
$$
and, finally,
$$
T_n(\z,E_s)=e^{inz/2}\frac{(w^{1/2})^n+(w^{-1/2})^n} 2=e^{inz/2}\mathfrak{T}_n\left(\sec \frac s 4\cos \frac z 2\right).
$$
Due to the fixed normalizations $c_*=0$ and $T_n(1,E_s)>0$, this extremal polynomial is unique, see Remarks \ref{rem22a} and \ref{rem22}. Generally, we can choose an arbitrary normalization point $c_*=c_0\in [0,2\pi)$
and   multiply the extremal polynomial by a unimodular constant. Thus, we obtain \eqref{16sep3}.
\end{proof}

\vskip 0.4cm

\section*{Acknowledgment}

S. Tikhonov was partially supported  by  MTM 2017-87409-P,  2017 SGR 358, and
 the CERCA Programme of the Generalitat de Catalunya.
 P. Yuditskii was supported by the Austrian Science Fund FWF, project no:  P29363-N32.
 The authors would like to thank the organizers of the
IX Jaen Conference on Approximation Theory (Ubeda, Jaen, Spain), 
where
a part of this work was carried out.


\end{document}